\documentclass[11pt]{amsart}

\usepackage{amsmath}

\newtheorem{thm}{Theorem}[section]
\newtheorem{lem}[thm]{Lemma}
\newtheorem{prop}[thm]{Proposition}

\newtheorem{remark}{Remark}

\newcommand{\fp}{\mathbb{F}_p}
\newcommand{\muh}{{\mu_H}}

\newcommand{\nuhat}{{\widehat{\nu}}}
\newcommand{\muhat}{{\widehat{\mu}}}
\newcommand{\nuk}{{\nu_k}}
\newcommand{\C}{{\mathbb{C}}}
\newcommand{\Z}{{\mathbb{Z}}}
\newcommand{\N}{{\mathbb{N}}}

\newcommand{\norm}[1]{\|#1\|_2}   
\newcommand{\supnorm}[1]{\|#1\|_\infty}   

\begin{document}
\title[Exponential sums over small multiplicative subgroups]{Bounds on  exponential sums over small multiplicative
  subgroups}
\author{P\"ar Kurlberg}

\address{Department of Mathematics, Royal Institute of Technology,
SE-100 44 Stockholm, Sweden}
\email{kurlberg@math.kth.se}

\date{May 20, 2007}
\thanks{The author was partially supported by grants from the G\"oran
Gustafsson Foundation, the Royal Swedish Academy of Sciences, and the
Swedish Research Council.}

\begin{abstract}
We show that there is significant cancellation in certain
exponential sums over small multiplicative subgroups of finite
fields, giving an exposition of the arguments by Bourgain and Chang
\cite{boch06}.
\end{abstract}
\maketitle


\section{Introduction}

Let $\psi : \fp \to \C$ be any non-trivial additive character in
$\fp$ (that is, $\psi(x)= \exp \left(\frac{2 \pi i x \xi}{p}
\right)$ for all $x\in \fp$, for some $\xi \in \fp^\times$), and
let $H$ be a subset of $\fp$. We are interested in obtaining good
upper bounds for
$$
\left| \sum_{x \in H} \psi(x)  \right| ;
$$
that is, significantly smaller than $|H|$. A traditional analytic
number theory approach when  $H$ is the multiplicative subgroup of
$\fp$ of index $m$ is to ``complete the sum'':\ We have
$$
\frac 1m \ \sum_{ \substack{\chi \pmod p \\ \chi^m=\chi_0 }}
\chi(n) = \begin{cases} 1 & \text{if}\ n\in H, \\ 0 &
\text{otherwise;}
\end{cases}
$$
where the sum runs through the Dirichlet characters $\pmod p$ with
order dividing $m$. Therefore
$$
\sum_{x \in H} \psi(x) =  \sum_{n \in \fp} \psi(n)\ \frac 1m \
\sum_{ \substack{\chi: \\ \chi^m=\chi_0 }} \chi(n) = \frac 1m \
\sum_{ \substack{\chi: \\ \chi^m=\chi_0 }} \sum_{n \in \fp}
\psi(n) \chi(n) .
$$
The last sum, $\sum_{n \in \fp} \psi(n) \chi(n)$, is  a Gauss sum
when $\chi\ne \chi_0$ and is known to have absolute value
$\sqrt{p}$; and $\sum_{n \in \fp} \psi(n) \chi_0(n)=-1$. We deduce
that
$$
\left| \sum_{x \in H} \psi(x)  \right| < \sqrt{p}.
$$
This is non-trivial when $H$ has substantially more than $p^{1/2}$
elements and classical arguments can sometimes give non-trivial bounds
for interesting sets $H$ as small as $p^{1/4}$, but not much smaller.
For $H$ a multiplicative subgroup, the first bound of the form
$\sum_{x \in H} \psi(x) \ll_\delta p^{-\delta}|H|$ with $\delta>0$ and
for $|H|$ significantly smaller than $p^{1/2}$ was obtained when $|H|
\gg_\epsilon p^{3/7+\epsilon}$ (for all $\epsilon>0$) by Shparlinski
\cite{shp91-gauss-sums}, and later refined to $|H| \gg_\epsilon
p^{3/8+\epsilon}$ by Konyagin and Shparlinski (unpublished), for $|H|
\gg_\epsilon p^{1/3+\epsilon}$ by Heath-Brown and Konyagin 
\cite{heko00-gauss-sum}, and for $|H| \gg_\epsilon p^{1/4+\epsilon}$
by Konyagin \cite{kon02-gauss-sums}.  An essential ingredient in
these results are upper bounds on the number of $\fp$-points on
certain curves/varieties that significantly go beyond what the Weil
bounds give.

In several recent articles Bourgain along with Chang, Glibichuk, and,
Konyagin showed how to get non-trivial upper bounds for various
interesting $H$ that are much smaller, using completely different
methods --- the techniques of additive combinatorics.  The aim of this
note is to give an exposition of these ideas in the simplest
case\footnote{See Section~\ref{sec:incomplete-sums} for an easy
extension to the case of incomplete sums.} by
showing that there is significant cancellation in such exponential
sums over small multiplicative subgroups $H$ of the finite field
$\fp$.  
\begin{thm}
\label{t:main} Given $\alpha > 0$, there exists $\beta =
\beta(\alpha) > 0$ such that if $|H| > p^\alpha$,
and $H$ is a multiplicative subgroup of  $\fp$,
then
\begin{equation}
  \label{eq:basic-exp-sum}
\sum_{x \in H} \psi( x) \ll p^{-\beta}|H| .
\end{equation}
\end{thm}

A proof of this result was first sketched by
Bourgain and Konyagin in \cite{boko03}, and  detailed proofs were
subsequently given
by Bourgain, Glibichuk, and Konyagin in \cite{boglko06}.
This note is based on the arguments by Bourgain and Chang in
\cite{boch06}, and is a somewhat streamlined version of notes from
a lecture series  given at KTH.

However, as alluded to above, the idea of using additive combinatorics
is very versatile.  For instance, in \cite{bou04-diffie,bou05-diffie}
Bourgain showed that under certain circumstances it is enough to
assume that $H$ has a small multiplicative doubling set, i.e., that
$|H \cdot H| < |H|^{1+\tau}$ 
for $\tau>0$ small.  In particular, one can take $H = \{ g^t : t_0
\leq t \leq t_1 \}$ as long as the multiplicative order of $g$ modulo
$p$ and $t_1-t_0$ are not too small, and thus it is also possible to
non-trivially bound incomplete exponential sums over small (as well as
large) multiplicative subgroups.
Further, by suitably generalizing the sum-product theorem to subsets
of $\fp \times \fp$ (some care is required since there are subsets of
$\fp \times \fp$, e.g., any line passing through $(0,0)$, that violate
a naive generalization of the sum-product theorem), Bourgain showed
that there is considerable cancellation in sums of the form
$\sum_{s_1=1}^t | \sum_{s_2=1}^t \psi( a g^{s_1} + b g^{s_1 s_2})|$
(consequently proving equidistribution for so-called Diffie-Hellman
triples in $\fp^3$) and in \cite{bou04a,bou05} he obtained bounds for
Mordell type exponential sums $\sum_{x=1}^p \psi(f(x))$, where $f(x) =
\sum_{i=1}^r a_i x^{k_i}$ is a sparse polynomial (under suitable
conditions on the $k_i$'s.)
Moreover, in \cite{Bourgain-Chang-gauss-sum,boch06}
Bourgain and Chang obtained bounds on sums over multiplicative
subgroups (and ``almost subgroups'') of general finite fields
$\mathbb{F}_{p^n}$, respectively $\Z/q\Z$ where $q$ is allowed to be
composite, but with a bounded number of prime divisors.

\subsection{A brief outline of the argument}
Define an $H$-invariant probability measure $\muh$ on $\fp$ by
$$
\muh(x) := \begin{cases} 1/|H| & \text{if}\ x\in H,\cr 0 &
\text{otherwise,}
\end{cases}
$$
and assume that \eqref{eq:basic-exp-sum} is violated, i.e., that
there exists $\xi \in \fp^\times$ for which
\begin{equation}\label{eq:large-fourier-coeff}
\muhat_H(\xi) = \sum_{x \in \fp} \muh(x)
 \exp \left(\frac{2 \pi i x \xi}{p} \right)
> p^{-\beta} .
\end{equation}

Let $\nu = \muh \ast \muh^-$, where $\muh^-(x) = \muh(-x)$, and
let $\nu_k$ be the $k$-fold convolution of $\nu$.  Using
\eqref{eq:large-fourier-coeff}, it is possible to show  (see
Proposition~\ref{p:smearing-out}) that for some tiny $\eta$ and
$k$ sufficiently large,
\begin{equation}
\label{eq:brief-statistical-mult-inv}
\sum_{x, \xi \in \fp}
|\nuhat_k(\xi)|^2 |\nuhat_k(x \xi)|^2 \nu_k(x)
>
p^{-10\eta} \sum_{\xi \in \fp} |\nuhat_k(\xi)|^2,
\end{equation}
and that the support of $\nuhat_k$ is essentially contained in
the set of ``large Fourier coefficients'' $\Lambda_\delta$
(cf. Proposition~\ref{p:finding-Lambda-sub-delta}.)
Now, 
$\nuhat_k$ being essentially supported on $\Lambda_\delta$ means that 
$\nuhat_k$ and $\nuhat_{2k}$ are ``similar'' (note that 
$\nuhat_{2k}(\xi) = \nuhat_k(\xi)^2$, and $\nuhat_k(\xi)\geq 0$ for all
$\xi$), 
hence  $\nu_k$ and $\nu_{2k}= \nu_k \ast \nu_k$ are also similar,
and this  might be seen as a form of statistical, or approximate,
additive invariance 
for the
measure $\nu_k$.
Further, by Parseval, (\ref{eq:brief-statistical-mult-inv}) says that 
$
\sum_{x,y \in \fp}
\nu_{2k}(y) \nu_{2k}(x^{-1} y) 
\nu_k(x)
>
p^{-10\eta} \sum_{x \in \fp} \nu_k(x)^2,
$
which we may interpret as $\sum_{ y \in \fp} \nu_{2k}(y) \nu_{2k}(x^{-1}
y)$ being correlated with $\nu_k$, and this in turn might be seen as 
statistical  multiplicative invariance.  
(Also see Remarks \ref{rem:statistical-additive-stability} and
\ref{rem:statistical-multiplicativ-stability}.)
%
With $S_1$ being the set of points assigned large relative mass (i.e.,
those $x$ for which 
$\nu_k(x)$ is close to $\supnorm{\nu_k}$)
 as a starting point, these 
invariance properties can then be used to find a subset of $S_1$ with
both small sum 
and product sets.
More precisely, using (\ref{eq:brief-statistical-mult-inv}), together with the
Balog-Gowers-Szemer\'edi theorem (cf. 
Theorem~\ref{t:bgs}) in multiplicative form, we can find a fairly
large subset $S_3 \subset S_1$ with a small product set.  Using the
Balog-Gowers-Szemer\'edi theorem again, but in additive form, we
then find a large subset $S_4 \subset S_3$ which has a small sum
set.  Now, since $S_4 \subset S_3$, $S_4$ also has a small product
set, hence it contradicts the sum-product theorem (cf.
Theorem~\ref{t:sum-product}.)

{\bf Acknowledgment:} It is my pleasure to thank John B. Friedlander
and Andrew Granville for their encouragement, as well as many helpful
comments and suggestions.  I am also grateful to University of Toronto
for its hospitality during my visit in April 2007, during which 
parts of this note were written up.

\section{Some additive combinatorics results}
\label{sec:addit-comb-backgr}

We will need two essential ingredients from additive combinatorics.
First we recall the sum-product theorem for
subsets of $\fp$, due to Bourgain, Katz and Tao \cite{bokata04} 
(for an expository note, see \cite{gre05-sum-product}.)
\begin{thm}
\label{t:sum-product} For any $\epsilon>0$ there exists
$\delta=\delta(\epsilon)$ such that the following holds: If $A
\subset \fp$ is a subset for which $p^\epsilon < |A| <
p^{1-\epsilon}$ then
$$
|A + A|+|A \cdot A| \gg |A|^{1+\delta} .
$$
\end{thm}


We will also need the following version of the Balog-Gowers-Szemer\'edi
theorem (this version of Theorem BGS' in   \cite{boch06} is an
immediate consequence of Theorem 5 in Balog's article herein
\cite{bal07}):

\begin{thm}\label{t:bgs}
Let $A$ and $B$ be finite subsets of an additive abelian group, $Z$,
and $G$ be a subset of $A\times B$, and let $S=\{a+b: (a,b)\in G\}$. If
$|A|, |B|, |S| \leq N$ and $|G|\geq \alpha N^2$ then there is an
$A'\subset A$ such that
\begin{equation}
\text{i) } |A'+A'|\leq \frac{2^{37}}{\alpha^8}N,\quad \text{ii) }
|A'|\geq \frac{\alpha^4}{2^{15}}N.
\end{equation}
\end{thm}

\section{The main technical result}
\label{sec:proof-of-prop-2.1}

In this section we prove the key technical result (cf.
\cite{boch06}, Proposition~2.1.):

\begin{prop}
\label{p:proposition2.1} Let $\mu$ be a probability measure on
$\fp$. If there exists a constant $\Delta\in (0, \frac 12]$ such
that
\begin{equation}
  \label{eq:2.5}
\sum_{\xi, y \in \fp} |\muhat(\xi)|^2 |\muhat(y\xi)|^2 \mu(y)
>
\Delta \sum_{\xi \in \fp}  |\muhat(\xi)|^2,
\end{equation}
and
\begin{equation}
  \label{eq:2.6}
\mu(0),\ \sum_{x \in \fp} \mu(x)^2 < \Delta/4
\end{equation}
then there exist a subset $S \subset \fp^\times$ such that
\begin{equation}
  \label{eq:size-of-S-and-l2-norm}
 \frac{\Delta^{254}}{2^{900}} \ p < |S|\ \sum_{\xi \in \fp}  |\muhat(\xi)|^2 <
  \frac 8 \Delta \ p,
\end{equation}
and
$$
|S+S| + |S \cdot S| <   \frac{2^{2729}}{\Delta^{768}} |S| .
$$
\end{prop}

To prove Proposition \ref{p:proposition2.1}  we will construct a
sequence of subsets $\fp \supset S_1 \supset S_2 \supset S_3 \supset
S_4$ such that $|S_i|/|S_{i+1}|=\Delta^{O(1)}$, where $S_3$ has a
small product set and $S_4$ has a small sum set.

First let us  recall some useful properties of
the finite Fourier transform.  For a given probability measure $\mu$
on $\fp$ define its Fourier 
transform to be
$$
\muhat(\xi) := \sum_{x \in \fp} \mu(x) \psi(x \xi) ,
$$
so that $\overline{\muhat(\xi)}=\muhat(-\xi)$. With this
normalization, Parseval's formula reads as
$$
p \sum_{x \in \fp} |\mu(x)|^2 =  \sum_{\xi \in \fp} |\muhat(\xi)|^2.
$$
As $\mu$ is a probability measure, we see that
$$
\phi(x) := p (\mu \ast \mu^-)(x) = \sum_{\xi \in \fp}
|\muhat(\xi)|^2 \psi(x \xi)
$$
is  $\geq 0$ for all $x$. We will replace the middle term in
\eqref{eq:size-of-S-and-l2-norm} by $|S|\phi(0)$. Moreover,
$$
\sum_{x \in \fp} \phi(x) = p,
$$
since $\mu \ast \mu^-$ is also a probability measure. From the
Fourier expansion of $\phi$, we have
\begin{equation}
\label{eq:phi-zero-bound} \max_{x \in \fp}\ \phi(x) =\phi(0) = p
\cdot (\mu \ast \mu^-)(0) = p \sum_x \mu(x)^2 \leq  \Delta p /4
\end{equation}
by \eqref{eq:2.6}.

\subsection{Multiplicative stability} We obtain the
following form of ``statistical multiplicative stability''.

\begin{lem} If \eqref{eq:2.5} and \eqref{eq:2.6} hold then
\begin{equation}
\label{eq:2.9} \sum_{x \in \fp} \sum_{y \in \fp^\times} \phi(x)
\phi(xy) \mu(y)
> \frac 34 \ \Delta p\phi(0)
\end{equation}
\end{lem}
\begin{proof}
For $y$ fixed, we have
$$
\sum_{x \in \fp} \phi(x) \phi(xy) = \sum_{\xi, \tau \in \fp}
|\muhat(\xi)|^2 |\muhat(\tau)|^2 \sum_{x \in \fp}
\psi(x\tau+xy\xi) = p \sum_{\xi \in \fp} |\muhat(\xi)|^2
|\muhat(-y \xi)|^2.
$$
Summing this over all $y \in \fp^\times$, we see that  the left hand
side of \eqref{eq:2.9}  equals
$$
p \sum_{y,\xi \in \fp} |\muhat(\xi)|^2 |\muhat(-y \xi)|^2 \mu(y) - p
\sum_{\xi \in \fp} |\muhat(\xi)|^2 |\muhat(0)|^2 \mu(0)
$$
$$
\geq p\ \Delta \sum_{\xi \in \fp} |\muhat(\xi)|^2 -p (\Delta/4)
|\muhat(0)|^2 \sum_{\xi \in \fp} |\muhat(\xi)|^2
$$
by \eqref{eq:2.5} and \eqref{eq:2.6}, as  $|\muhat(-y\xi)|^2 =
|\muhat(y\xi)|^2$, which yields the result since  $|\muhat(0)|^2
\leq 1$.
\end{proof}
\begin{remark}
Note that $\sum_{x \in \fp} \sum_{y \in \fp^\times} \phi(x) \phi(xy)
\mu(y) \leq \phi(0) \sum_{x,y \in \fp} \phi(x) \mu(y) \leq p
\phi(0)$.  In our applications, we shall take $\Delta =p^{-\epsilon}$,
and for this choice of $\Delta$, the lower bound~\eqref{eq:2.9} is
fairly good. 
\end{remark}

As a starting point for a  multiplicatively stable subset, we use
the points which are assigned large measure by  $\mu \ast \mu^-$.

\begin{lem} If \eqref{eq:2.5} and \eqref{eq:2.6} hold and
$$
S_1 := \{ x \in \fp : \phi(x) > \frac{1}{8} \Delta \phi(0) \}
$$
then
\begin{equation} \label{eq:2.11}
\sum_{\substack{x \in S_1, y \in \fp^\times \\ xy \in S_1}} \phi(x)
\phi(xy) \mu(y) > \frac 12 \ \Delta p \phi(0)
\end{equation}
\end{lem}
\begin{proof}
We have
$$
\sum_{\substack{x \in S_1, y \in \fp^\times \\ xy \in S_1}} \geq
\sum_{\substack{x \in \fp, y \in \fp^\times }} - \sum_{\substack{x
\in \fp \backslash S_1, y \in \fp^\times}} - \sum_{\substack{x \in
\fp, y \in \fp^\times \\ xy \not \in S_1}}.
$$
By \eqref{eq:2.9}, the first term on the right hand side is $>
(3/4)\Delta p \phi(0)$.  The second term
$$
\sum_{\substack{x \in \fp \backslash S_1, y \in \fp^\times}} \phi(x)
\phi(xy) \mu(y)
$$
is, since $\phi(x) \leq \Delta\phi(0)/8$ for $x \not \in S_1$,
bounded by
$$
\frac{\Delta\phi(0)}{8} \sum_{\substack{x \in \fp \backslash S_1, y
\in \fp^\times}} \phi(xy) \mu(y) \leq \frac{ \Delta\phi(0)}{8}
\sum_{y \in \fp^\times} \mu(y) \sum_{x \in \fp} \phi(xy) \leq \frac{
\Delta p\phi(0)} {8}
$$
since $\sum_{x \in \fp} \phi(xy) = p$ for $y \neq 0$ and $\mu$ is a
probability measure.  Similarly, the third term is bounded by $\Delta
p \phi(0)/8$, hence the left hand side of \eqref{eq:2.11} is $>
\Delta p\phi(0)(3/4-1/8-1/8) \geq \Delta p\phi(0)/2$.
\end{proof}

We proceed to estimate the size of $S_1$.

\begin{lem} If \eqref{eq:2.5} and \eqref{eq:2.6} hold then
\label{l:size-of-s1}
\begin{equation}
\label{eq:2.12} \frac{\Delta p}{2\phi(0)} < |S_1| < \frac{8
p}{\Delta \phi(0)} .
\end{equation}
Moreover, if we let
$$S_2 := S_1 \backslash \{ 0 \}\subset \fp^\times,$$ 
then $|S_2| \geq |S_1|/2$.
\end{lem}

\begin{proof} For the lower bound, note that
\begin{equation}
\label{eq:2.13} |S_1| = \sum_{y \in \fp} |S_1| \mu(y) \geq \sum_{y
\in \fp^\times} |S_1 \cap y^{-1}S_1| \mu(y) = \sum_{\substack{x \in
S_1, y \in \fp^\times \\ xy \in S_1} } \mu(y)
\end{equation}
$$
\geq \frac{1}{\phi(0)^2} \sum_{\substack{x \in S_1, y \in \fp^\times
\\ xy \in S_1} } \phi(x)\phi(xy) \mu(y) > \frac{\Delta p}{2\phi(0)}
$$
by \eqref{eq:2.11}, which is $\geq 2$ by \eqref{eq:phi-zero-bound},
so that $|S_2|\geq |S_1|/2$. For the upper bound, note that
$$
|S_1| < \frac{8 }{\Delta\phi(0)} \sum_{x \in S_1} \phi(x) \leq
\frac{8}{\Delta\phi(0)} \sum_{x \in \fp} \phi(x) = \frac{8 p}{\Delta
\phi(0)}.
$$
\end{proof}

To show that there are many $y$ such that $|S_2 \cap y^{-1}S_2|$ is
fairly large, we begin by giving a lower bound on the expected size
of the intersection.

\begin{lem} If \eqref{eq:2.5} and \eqref{eq:2.6} hold then
\begin{equation}
  \label{eq:2.18}
\sum_{y \in \fp^\times} |S_2 \cap y^{-1}S_2| \mu(y) \geq
\frac{\Delta p}{4 \phi(0)}
\end{equation}
\end{lem}
\begin{proof} Since $S_2 \cap y^{-1}S_2 = (S_1 \cap y^{-1}S_1) \backslash \{ 0
\}$ for all $y \in \fp^\times$ we have
$$
\sum_{y \in \fp^\times} |S_2 \cap y^{-1}S_2| \mu(y) \geq \sum_{y \in
\fp^\times} |S_1 \cap y^{-1}S_1| \mu(y) - \sum_{y \in \fp^\times}
\mu(y)
$$
$$
> \frac{\Delta p}{2 \phi(0)} -
1\geq \frac{\Delta p}{4 \phi(0)}
$$
by the right hand side of \eqref{eq:2.13} and as $\sum_{y \in \fp}
\mu(y) = 1$, and then by \eqref{eq:phi-zero-bound}.

\end{proof}

In the next result we show that there are many $y$ for which $|S_2
\cap y^{-1}S_2|$ is large:

\begin{lem} If \eqref{eq:2.5} and \eqref{eq:2.6} hold and
\begin{equation}
\label{eq:2.19} T := \left\{ y \in \fp^\times : |S_2 \cap y^{-1}
S_2|
> \frac{\Delta p}{8 \phi(0)} \right\}
\end{equation}
then
\begin{equation}
\label{eq:2.20} |T| \geq \frac{\Delta^5}{2^{15}} |S_1|
\end{equation}
\end{lem}
\begin{proof}
$$
|S_2| \mu(T) = |S_2| \sum_{y \in T}\mu(y) \geq \sum_{y \in T} |S_2
\cap y^{-1} S_2| \mu(y)
$$
\begin{equation}
  \label{eq:2.21}
= \sum_{y \in \fp^\times} |S_2 \cap y^{-1} S_2| \mu(y) - \sum_{y \in
\fp^\times \backslash T} |S_2 \cap y^{-1} S_2| \mu(y) \geq
\frac{\Delta p}{8 \phi(0)} > \frac{\Delta^2}{64} |S_2|
\end{equation}
by \eqref{eq:2.18} and from the definition of $T$, and then by
\eqref{eq:2.12}  and the trivial bound $|S_2| \leq |S_1|$, so that
$\mu(T)>\Delta^2/64$.

On the other hand, by Cauchy-Schwartz and Parseval's identity,
$$
\mu(T) \leq |T|^{1/2} \left( \sum_{x \in T} \mu(x)^2 \right)^{1/2}
$$
$$
\leq |T|^{1/2} \left( \frac{1}{p} \sum_{\xi \in \fp} |\muhat(\xi)|^2
\right)^{1/2} = \left( \frac{|T| \phi(0)}{p} \right)^{1/2},
$$
so that $|T|  \geq p\Delta^4/(2^{12}\phi(0))
>(\Delta^5/2^{15})|S_1|$, by \eqref{eq:2.12}.
\end{proof}

Thus, by shrinking $T$ if necessary, we have found a set $T$ such
that
$$
 (\Delta^5/2^{15}) |S_2| \leq |T| \leq |S_2|
$$
with the property that for all $y \in T$,
\begin{equation}
  \label{eq:intersections-bound}
|S_2 \cap y^{-1}S_2| > \frac{\Delta p}{8 \phi(0)}
>\frac{ \Delta^2}{2^6} |S_1|
\geq \frac{ \Delta^2}{2^6} |S_2|
\end{equation}
by \eqref{eq:2.12}.

Let $ G := \{(x,y) : x \in S_2, y \in T, xy \in S_2 \} \subset S_2
\times T \subset \fp^\times \times \fp^\times$.  By
(\ref{eq:intersections-bound}), the number of $x$ such that $(x,y)
\in G$ is at least $2^{-6} \Delta^2 |S_2|$ for each $y \in T$.
Therefore, since $|T|\geq 2^{-15} \Delta^{5} |S_2|$, we find that
$$
|G|  \geq 2^{-6} \Delta^2 |S_2| \cdot 2^{-15} \Delta^{5} |S_2| =
(\Delta/8)^{7} |S_2|^2.
$$
By the definition of $G$ we know that
$$
\{ st: \ (s,t) \in G\} \subset S_2;
$$
so, with $g$ a primitive root modulo $p$ and defining
  $\log_{g,p}(s)$ to be the smallest integer $m \geq 0$ such that $g^m
  \equiv s \mod p$, and  by taking
$A=\{ \log_{g,p} s:\ s\in S_2\},\ B=\{ \log_{g,p} t:\ t\in T\}$ with 
$N=|S_2|$ and $\alpha=(\Delta/8)^7$ in Theorem~\ref{t:bgs}, we
obtain a subset $A'$ of $A$, with $|A'| > (\Delta^{28}/2^{99}) |A|$,
for which
$$
|A'+A'| \leq (2^{205}/\Delta^{56}) N < (2^{304}/\Delta^{84})  |A'| .
$$
Therefore $S_3 =\{ g^a:\ a\in A'\}$ is a subset of $S_2$ for which
\begin{equation}
\label{eq:s3-vs-s1-size} |S_3| > (\Delta^{28}/2^{100}) |S_1| ,
\end{equation}
by Lemma~\ref{l:size-of-s1}, and
$$
|S_3 \cdot S_3| \leq (2^{304}/\Delta^{84}) |S_3| .
$$

\subsection{Additive stability}
We finish the proof of Proposition \ref{p:proposition2.1} by finding
a subset $S_4$ of $S_3$ with a small sum set. We first show that
$S_3$ exhibits ``statistical additive stability''; to do this we
only need to use that $S_3 \subset S_1$, together with the
definition of $S_1$.
\begin{lem}  If \eqref{eq:2.5} and \eqref{eq:2.6} hold then
\begin{equation}
\label{eq:2.25} \sum_{x_1, x_2 \in S_3} \phi(x_1 - x_2)
>
2^{-6} \Delta^2 \phi(0) |S_3|^2
\end{equation}

\end{lem}
\begin{proof}
Recalling that $\phi(x) = p (\mu * \mu^-)(x)$, we find, using the
Cauchy-Schwarz inequality, that
$$
\left( \frac 1p \sum_{x \in S_3} \phi(x) \right)^2= \left(
 \sum_{y \in \fp} \mu(y) \sum_{x \in S_3} \mu(x+y) \right)^2 \leq
\sum_{y \in \fp} \mu(y)^2 \cdot \sum_{y \in \fp} \left( \sum_{x \in
S_3} \mu(x+y) \right)^2
$$
$$
 = \frac{\phi(0)}p \sum_{x_1, x_2 \in S_3}
\sum_{y \in \fp} \mu(x_1+y) \mu(x_2+y) = \frac{\phi(0)}{p^2}
\sum_{x_1, x_2 \in S_3} \phi(x_1 - x_2) .
$$
Now $\sum_{x \in S_3} \phi(x) > \frac{\Delta}{8} \phi(0) |S_3|$,
since $S_3 \subset S_1$, and the lemma follows.
\end{proof}

To obtain an additively stable subset we will, as before, use
Theorem~\ref{t:bgs}.  First, let
\begin{equation}
\label{eq:2.27} S_0 := \{ x \in \fp : \phi(x) > 2^{-7} \Delta^2
\phi(0) \}
\end{equation}
Then
$$
|S_0| \leq \frac{2^{7}}{\Delta^2\phi(0)} \sum_{x \in S_0} \phi(x)
\leq \frac{2^{7} p}{\Delta^2\phi(0)}  \leq \frac{2^{8}}{\Delta^3}
|S_1| < \frac{2^{108}}{\Delta^{31}} |S_3|
$$
by (\ref{eq:2.12}) and then (\ref{eq:s3-vs-s1-size}).

Using $S_0, S_3$ we can now define a fairly large graph $G'$.
\begin{lem} If \eqref{eq:2.5} and \eqref{eq:2.6} hold then
$$
G' := \{ (x_1, -x_2) \in S_3 \times (-S_3) : x_1 - x_2 \in S_0 \}
\subset S_3 \times (-S_3).
$$
has at least $2^{-7} \Delta^2 |S_3|^2$ elements.
\end{lem}
\begin{proof}
We have
$$
|G'| \cdot \phi(0) \geq \sum_{ (x_1, -x_2) \in G'} \phi(x_1 - x_2)
$$
$$
= \sum_{ x_1, x_2 \in S_3} \phi(x_1 - x_2) - \sum_{ (x_1, -x_2) \in
S_3 \times (-S_3) \backslash G'} \phi(x_1 - x_2)
$$
$$
\geq 2^{-6} \Delta^2 \phi(0) |S_3|^2-2^{-7} \Delta^2 \phi(0)|S_3|^2
$$
by \eqref{eq:2.25} and \eqref{eq:2.27}, and the result follows.
\end{proof}

Since $\{ x_1 - x_2:\  (x_1, -x_2) \in G'\} \subset S_0$ we can
apply Theorem~\ref{t:bgs} with $A=S_3,\ B=-S_3,\ G=G',\
N=(2^{108}/\Delta^{31})|S_3|$ and $\alpha = \Delta^{64}/2^{223}$ to
obtain a subset $S_4 \subset S_3$ with
\begin{equation}
\label{eq:s3-vs-s4-size} |S_4| > \frac{\Delta^{256}}{2^{907}}N =
\frac{\Delta^{225}}{2^{799}} |S_3|
\end{equation}
for which
$$
|S_4+S_4|  < \frac{2^{1821}}{\Delta^{512}} N =
\frac{2^{1929}}{\Delta^{543}}  |S_3| < \frac{2^{2728}}{\Delta^{768}}
|S_4| .
$$
Moreover, since $S_4 \subset S_3$, we find that
$$
|S_4 \cdot S_4| \leq |S_3 \cdot S_3| < (2^{304}/\Delta^{84})  |S_3|
 < (2^{1103}/\Delta^{309})  |S_4| .
$$
Finally, by \eqref{eq:2.12}, then (\ref{eq:s3-vs-s4-size}),
(\ref{eq:s3-vs-s1-size}), and Lemma~\ref{l:size-of-s1}, we have
$$
\frac{8p}{\Delta\phi(0)} > |S_1|\geq |S_4| >
\frac{\Delta^{225}}{2^{799}} |S_3| > \frac{\Delta^{253}}{2^{899}}
 |S_1|   > \frac{\Delta^{254}}{2^{900}}  \  \frac{p}{ \phi(0)}
.
$$
Taking $S=S_4$ we have found a set with the desired properties.

\section{Proof of Theorem~\ref{t:main}}
\label{sec:proof-theorem}
\subsection{Preliminaries}
\label{subsec:prelims}

Let $\mu$ be a given probability measure on $\fp$.  Recall that the
Fourier transform of $\mu$ was defined to be $\muhat(\xi) := \sum_{x \in
  \fp} \mu(x) \psi(x \xi)$, and hence
$\overline{\muhat(\xi)}=\muhat(-\xi)$.  With this normalization,
Parseval's formula reads as $ p \sum_{x \in \fp} |\mu(x)|^2 =
\sum_{\xi \in \fp} |\muhat(\xi)|^2$.  Moreover, if $\nu$ is another
probability measure then
$$
\sum_{x \in \fp} \mu(x) \nuhat(x) = \sum_{\xi \in \fp}
\overline{\muhat(\xi)} \nu(-\xi) = \sum_{\xi \in \fp} \muhat(-\xi)
\nu(-\xi) = \sum_{\xi \in \fp} \muhat(\xi) \nu(\xi)
$$

Let $\nu := \mu \ast \mu^-$, that is $\nu(x) = \sum_{y,z: y-z=x}
\mu(y) \mu(z)$, so that $\nu(-x)=\nu(x)$ and $\nuhat(x) =
|\muhat(x)|^2$. If $\nu_k$ is the $k$-fold convolution of $\nu$,
that is
$$
\nu_k(x) :=
\sum_{\substack{y_1, y_2, \ldots y_k \in \fp \\
y_1+ y_2+ \ldots +y_k = x }} \nu(y_1)  \nu(y_2) \cdots \nu(y_k),
$$
then $\nuhat_k(x) = |\muhat(x)|^{2k} \geq 0$. Notice that
$\nu(x)=\sum_{y,z: y-z=x} \mu(y) \mu(z)\leq \max_z \mu(z)\sum_{y}
\mu(y) =\max_z \mu(z)$ for all $x$; and similarly
\begin{equation} \label{eq:bounded-sup-norm}
\max_x \nu_k(x) \leq \max_z \mu(z) \ \text{for all} \ k.
\end{equation}

We have
$$
\norm{\muh}^2 = \sum_{x \in \fp} |\muh(x)|^2 = 1/|H| .
$$
Note that $\muh(hx)=\muh(x)$ for all $h\in H$, and so
$\muhat_H(hx)=\muhat_H(x)$ for all $h\in H$, and
$\nu_k(hx)=\nu_k(x)$ for all $h\in H$ and $k\geq 1$.

\subsection{The set of large Fourier coefficients}

Given $\delta>0$, let
$$
\Lambda_\delta := \{ \xi \in \fp : |\muhat(\xi)| > p^{-\delta}\}
$$
be the set of ``large'' Fourier coefficients of $\mu$.

\begin{lem}
\label{l:lambda-delta-size} Suppose that $\mu=\muh$. We have
$$
|\Lambda_\delta| \leq p^{1+2\delta}/|H|.
$$
Also if $|\muhat_H(\xi)|>p^{-\delta}$ for some {\em nonzero} $\xi
\in \fp^\times$, then
$$
|\Lambda_\delta| \geq |H|.
$$
\end{lem}
\begin{proof} For any measure $\mu$ on $\fp$ we have
$$
|\Lambda_\delta| \leq p^{2\delta} \sum_{\xi \in \Lambda_\delta}
|\muhat(\xi)|^2 \leq p^{2\delta} \sum_{\xi \in \fp} |\muhat(\xi)|^2
=p^{1+2\delta} \sum_{x \in \fp} |\mu (x)|^2,
$$
and the first result follows since this last sum equals $1/|H|$ for
$\mu=\muh$. For the second result note that if
$\xi\in\Lambda_\delta$  then $|\muhat_H(h\xi)|=|\muhat_H(\xi)|
> p^{-\delta}$ for all $h\in H$, so that $ h\xi\in\Lambda_\delta$ for all $h\in H$.
\end{proof}

We will now show that it is possible to find $k,\delta$ so that the
support of $\nuhat_k$ is, in $L^2$-sense, essentially given by
$\Lambda_\delta$.

\begin{prop}
\label{p:finding-Lambda-sub-delta} For any measure $\mu$ on $\fp$,
where $p \geq 3$, and any $\eta \geq 5/(p^3 \log p)$, there exists an
integer $k \geq 4$ and
$$
\delta \in (0, \eta/k^2)
$$
such that
\begin{equation} \label{eq:lambda-delta-size-vs-l2-norm}
p^{-\eta} |\Lambda_\delta| \leq \sum_{\xi \in \fp}
|\nuhat_k(\xi)|^2 \leq p^{\eta} |\Lambda_\delta|
\end{equation}
and, in particular,
\begin{equation} \label{eq:l2-support-is-lambda-delta}
\sum_{\xi \in \fp} |\nuhat_k(\xi)|^2 \leq p^{2\eta} \sum_{\xi \in
\Lambda_\delta} |\nuhat_k(\xi)|^2.
\end{equation}
\end{prop}

\begin{proof}
For any $k \in \N$  we have
\begin{equation}  \label{eq:0.13}
\sum_{\xi \in \fp} |\nuhat_k(\xi)|^2 =   \sum_{\xi \in
\Lambda_{1/k}} |\nuhat_k(\xi)|^2 + \sum_{\xi \not \in \Lambda_{1/k}}
|\nuhat_k(\xi)|^2 \leq |\Lambda_{1/k}| + p (p^{-1/k})^{4k} =
|\Lambda_{1/k}|+1/p^3
\end{equation}
since each $\nuhat_k(\xi) \leq 1$.

We define a sequence of integers $k_0 = 4<k_1<\dots$ where
$k_{i+1} = [k_i^2/\eta]+1$ for each $i\geq 0$, and let  $\delta_i
= 1/k_{i+1}$ for each $i$. Note that $k_i^2/\eta <
k_{i+1}=1/\delta_i$ so that $k_i\delta_i<\eta/k_i\leq \eta/4$.
Since $\nuhat_{k_i}(\xi) = |\muhat_H(\xi)|^{2k_i}$, we have
$$
\sum_{\xi \in \Lambda_{\delta_i}} |\nuhat_{k_i}(\xi)|^2 >
|\Lambda_{\delta_i}| \cdot p^{-4k_i \delta_i} \geq
|\Lambda_{\delta_i}| \cdot p^{-\eta} .
$$
We note that the lower bound in
\eqref{eq:lambda-delta-size-vs-l2-norm} follows 
from this, as well as \eqref{eq:l2-support-is-lambda-delta}, once we
establish the upper bound in
\eqref{eq:lambda-delta-size-vs-l2-norm}.

Now, there exists an integer $i\in [0,M]$, where $M=2([1/\eta]+1)$,
such that $\sum_{\xi \in \fp} |\nuhat_{k_i}(\xi)|^2 \leq p^{\eta}
|\Lambda_{\delta_i}|$ else
$$
p^{\eta} |\Lambda_{1/k_{i+1}}| = p^{\eta} |\Lambda_{\delta_i}| <
\sum_{\xi \in \fp} |\nuhat_{k_i}(\xi)|^2 \leq |\Lambda_{1/k_i}| +
1/p^3 \leq |\Lambda_{1/k_i}|(1 + 1/p^3)
$$
for each $i$, by (\ref{eq:0.13}), and so
$$
|\Lambda_{1/k_M}| < p^{-M \eta} |\Lambda_{1/k_0}|(1 + 1/p^3)^M
\leq p^{1-M \eta} (1 + 1/p^3)^M \leq p^{-1} (1 + 1/p^3)^M<1
$$
since $M\leq \frac 12 p^3\log p$, which is untrue (as  $0 \in
\Lambda_{1/k}$ for all $k \in \N$).

We select $k=k_i$ and $\delta=\delta_i$.
\end{proof}

\begin{remark}
Note that the proof gives us $k\ll \exp(\exp(O(1/\eta)))$.
\end{remark}

\begin{remark}
\label{rem:statistical-additive-stability}
Since the support of $\nuhat_k$ is essentially given by
$\Lambda_\delta$, it is easy to see that the same holds for
$\nuhat_{2k}$; we may interpret this as $\nu_k \ast \nu_k$ being
``similar'' to $\nu_k$, and hence that  $\nu_k$ is ``approximately
additively stable''.
\end{remark}

In the following key Lemma, the $H$-invariance of $\muh$, and hence of
$\nuhat_k$, is essential.

\begin{lem}
\label{l:smear-out-lemma} For $\mu=\muh$ and all $\xi \in \fp$, we
have
$$
\nuhat_k(\xi)^{4k}
\leq
\sum_{x \in \fp}
\nuhat_k(x \xi)^2 \nuk(x)
$$
\end{lem}
\begin{proof}
The case $\xi=0$ is immediate, hence we may assume that $\xi \neq
0$. Now, since $\nuhat_k(h \xi)=\nuhat_k(\xi)$  for all $h\in H$, we
have
$$
\nuhat_k(\xi)^2 = \sum_{x \in \fp} \nuhat_k(x \xi)^2 \muh(x) =
\sum_{x \in \fp}  \nu_{2k}(- x \xi^{-1}) \muhat_H(x),
$$
by Parseval's formula. Now note that if $\mu$ is any probability
measure and $l \geq 1$, then $\sum_x \mu(x)f(x)\leq (\sum_x \mu(x)
|f(x)|^l)^{1/l}$. Therefore the above gives
$$
\nuhat_k(\xi)^{4k} \leq    \sum_{x \in \fp} \nu_{2k}(- x \xi^{-1})
|\muhat_H(x)|^{2k}  = \sum_{x \in \fp} \nu_{2k}(- x \xi^{-1})
\nuhat_k(x)
$$
since $|\muhat_H(x)|^{2k} = \nuhat(x)^k = \nuhat_k(x)$ and, applying
Parseval one more time, we obtain
$$
\nuhat_k(\xi)^{4k}
\leq
\sum_{x \in \fp} \nuhat_k(-x \xi)^2 \nu_k(-x)
=
\sum_{x \in \fp} \nuhat_k(x \xi)^2 \nu_k(x)
$$
\end{proof}

We consequently obtain:

\begin{prop}
\label{p:smearing-out} With $k,\eta$ as in
Proposition~\ref{p:finding-Lambda-sub-delta}, we have
$$
p^{-10\eta}\sum_{\xi \in \fp} \nuhat_k(\xi)^{2} \leq \sum_{\xi \in
\fp} \sum_{x \in \fp} \nuhat_k(\xi)^{2} \nuhat_k(x \xi)^2 \nuk(x)
$$
\end{prop}
\begin{proof}
By Proposition~\ref{p:finding-Lambda-sub-delta}, we have
$$
p^{-2\eta} \sum_{\xi \in \fp} \nuhat_k(\xi)^{2} \leq \sum_{\xi \in
\Lambda_\delta} \nuhat_k(\xi)^{2} \leq p^{8k^2\delta} \sum_{\xi
\in \Lambda_\delta} \nuhat_k(\xi)^{4k+2} \leq p^{8 \eta} \sum_{\xi
\in \fp} \nuhat_k(\xi)^{4k+2}
$$
which, by Lemma~\ref{l:smear-out-lemma}, is
$$
\leq p^{8 \eta} \sum_{\xi \in \fp} \sum_{x \in \fp}
\nuhat_k(\xi)^{2} \nuhat_k(x \xi)^2 \nuk(x).
$$
\end{proof}

\begin{remark}
\label{rem:statistical-multiplicativ-stability}
Since $\nuhat_k(x \xi) \leq 1$ and $\nu_k$ is a probability measure,
we find that
$\sum_{\xi,x \in \fp} \nuhat_k(\xi)^{2} \nuhat_k(x \xi)^2 \nuk(x) \leq
\sum_{\xi \in \fp} \nuhat_k(\xi)^{2}$, so the lower bound on the
double sum in Proposition~\ref{p:smearing-out} is quite good.
Further, using Parseval on the two sums over $\xi$ (ignoring the term
$x=0$) we find that  $\sum_{y \in \fp} \nu_{2k}(y)\nu_{2k}(yx^{-1})$,
which we can interpret as a multiplicative translate of $\nu_{2k}$
with itself, is highly correlated with $\nu_k(x)$.  Thus,  the Proposition
might be interpreted as a statement of ``approximate multiplicative
stability'' of $\nu_k$.  (Since the essential support of $\nuhat_k$ is
given by $\Lambda_\delta$, the same holds for $\nuhat_{2k}$, so in some
sense  $\nu_k$ and $\nu_{2k}$ are ``similar''.)
\end{remark}

To go from statistical additive/multiplicative stability to a subset
that contradicts the sum-product Theorem, we will apply Proposition
\ref{p:proposition2.1} with $\mu = \nu_k$ and $\Delta = p^{-10
\eta}$ (and note that \eqref{eq:bounded-sup-norm} implies
\eqref{eq:2.6} provided $1/|H|<\Delta/4$), and select $\delta$ and
$k$ as in Proposition \ref{p:finding-Lambda-sub-delta}. Assume that
$|\muhat_H(\xi)|>p^{-\delta}$ for some $\xi \in \fp^\times$. We thus
obtain a set $S$ such that
$$
|S+S| + |S \cdot S| <   2^{2729} p^{7680\eta} |S| .
$$
Note that
$$ p^{-\eta} |H| \leq p^{-\eta} |\Lambda_\delta|
\leq \sum_{\xi \in \fp} |\nuhat_k(\xi)|^2 \leq p^{\eta}
|\Lambda_\delta| \leq p^{1+\eta+ 2 \delta }/|H|
$$
by (\ref{eq:lambda-delta-size-vs-l2-norm}) and
Lemma~\ref{l:lambda-delta-size}, so that
\eqref{eq:size-of-S-and-l2-norm} gives, as $2\delta<\eta$,
$$
 \frac{1}{2^{900}} \  \frac{|H|}{p^{2542\eta}}< |S| <
 8  \ \frac{p^{1+11\eta}}{|H|}.
$$
Now select $\eta = \min\{ \alpha/6000, \delta(\alpha/2)/8000\}$, so
that the sum-product Theorem \ref{t:sum-product} is violated with
$\epsilon= \alpha/2$ for $p$ sufficiently large, and thus
$|\muhat_H(\xi)|\leq p^{-\delta}$ for all $\xi \in \fp^\times$. The
Theorem follows with $\beta=\delta\gg \exp(-\exp(C/\eta))$ for some
constant $C>0$.

\section{Incomplete sums}
\label{sec:incomplete-sums}
The proof of Theorem~\ref{t:main} can fairly easily be 
extended to incomplete sums over multiplicative subgroups.
\begin{thm}
Let $g \in \fp^\times$ have multiplicative order at least $T$, and let
$H = \{ g^t : 0 \leq t < T\}$.  If $|H| = T > p^\alpha$, then
$$
\sum_{x \in H}
\psi(x) \ll 
p^{-\beta} |H|
$$
\end{thm}
Define $\muh, \muhat_H, \nu_k, \Lambda_\delta$ etc as before.  To
obtain a 
contradiction, we will assume that $|\muhat_H(\xi_0 )| > 2
p^{-\delta}$ for some $\xi_0 \in \fp^\times$.

We begin by showing that $\Lambda_\delta$, the set of large Fourier
coefficients, is almost of size $|H|$, and that $\muhat$ is quite
large on $\Lambda_\delta \cdot H_1 $ for a fairly large subset $H_1 \subset H$.
\begin{lem}
Let
$$
H_1 := \{ g^t : 0 \leq t < |H| p^{-\delta}/4\}.
$$
If $|\muhat(\xi_0)|>2 p^{-\delta}$ for some $\xi_0 \in \fp^\times$, then 
$$
|\Lambda_\delta| \geq |H_1|
$$
Moreover, if $\xi \in \Lambda_\delta$ and $h \in H_1$, then 
$$
|\muhat_H(h \xi)| > |\muhat_H(\xi)|/2.
$$
\end{lem}
\begin{proof}
For $l \in \Z$ such that $0 \leq l < T$, we have
$$
\muhat_H(g^l \xi) =   
\sum_{x \in \fp} \psi( g^l \xi x) \muh(x) 
=
\sum_{x \in \fp} \psi( \xi x) \muh(g^{-l}x) 
=
\frac{1}{|H|} 
\sum_{x \in g^l H} 
\psi( \xi x) 
$$
$$
=
\frac{1}{|H|} 
\left(
\sum_{x \in  H} 
\psi( \xi x) 
+ 2 \theta l
\right)
$$ 
for some $\theta$ such that $|\theta|\leq 1$. 
Thus, if  $l < |H|p^{-\delta}/4$, then 
\begin{equation}
  \label{eq:muhat-translate-is-big}
|\muhat_H( g^l\xi)| > |\muhat_H( \xi)| -  p^{-\delta}/2. 
\end{equation}
In particular, if $h \in H_1$, then  $|\muhat_H( h \xi_0)| \geq
|\muhat_H(\xi_0)| - 
p^{-\delta}/2 >  2 p^{-\delta} - p^{-\delta}/2 > 
p^{-\delta}$ and hence $|\Lambda_\delta| \geq |H_1|$.
Finally, if $\xi \in \Lambda_\delta$ then
$|\muhat_H(\xi)|>p^{-\delta}$, so the second assertion follows from
\eqref{eq:muhat-translate-is-big}. 
\end{proof}

\begin{lem}
\label{l:substitute-one}
If $\xi \in \Lambda_\delta$, then 
$$
\nuhat_k(\xi)^{4k}
\leq 
2^{8k^2+6k} p^{2k \delta}
\sum_{x \in \fp}
\nuhat_k(h \xi)^2  
\nuk(x) 
$$  
\end{lem}
\begin{proof}
If $\xi \in \Lambda_\delta$, then $|\muhat_H(\xi h)| \geq |\muhat_H(\xi
)|/2$ for all $h \in H_1$.  Hence
$$
\nuhat_k(\xi)^2 \leq
\frac{2^{4k}}{|H_1|}
\sum_{h  \in H_1} \nuhat_k(h \xi)^2  
\leq
\frac{2^{4k}|H|}{|H_1|}
\sum_{x \in \fp}
\nuhat_k(h \xi)^2  
\muh(x)$$
$$
=
2^{4k+3}p^{\delta}
\sum_{x \in \fp}
\nuhat_k(h \xi)^2  
\muh(x)
$$
since $|H|/|H_1| \leq 8p^{\delta}$.  
Thus, if $\xi \in \Lambda_\delta$, then
$$
\nuhat_k(\xi)^{4k} \leq
2^{8k^2+6k} p^{2k \delta}
\left( \sum_{x \in \fp}
\nuhat_k(h \xi)^2  
\muh(x) \right)^{2k}
\leq 
2^{8k^2+6k} p^{2k \delta}
\sum_{x \in \fp}
\nuhat_k(h \xi)^2  
\nuk(x) 
$$
by the same argument used in the proof of
Lemma~\ref{l:smear-out-lemma}. 
\end{proof}
\begin{prop}
For $p$ sufficiently large, 
$$
p^{-11\eta}
\sum_{\xi \in \fp}
\nuhat_k(\xi)^2
\leq
\sum_{\xi,x \in \fp}
\nuhat_k(\xi)^2
\nuhat_k(\xi x)^2
\nu_k(x)
$$
\end{prop}
\begin{proof}
Arguing as in the proof of Proposition~\ref{p:smearing-out} find that 
$$
p^{-2\eta}
\sum_{\xi \in \fp}
\nuhat_k(\xi)^2
\leq 
\sum_{\xi \in \Lambda_\delta}
\nuhat_k(\xi)^2
\leq
p^{8 k^2 \delta}
\sum_{\xi \in \Lambda_\delta}
\nuhat_k(\xi)^{4k+2}
\leq 
p^{8\eta}
\sum_{\xi \in \Lambda_\delta}
\nuhat_k(\xi)^{4k+2}
$$  
which, by Lemma~\ref{l:substitute-one} is 
$$
\leq
p^{8 \eta + 2k\delta}
2^{8k^2+6k}
\sum_{\xi \in \Lambda_\delta}
\sum_{x \in \fp}
\nuhat_k(\xi)^{2}
\nuhat_k(\xi x)^{2}
\nuk(x)
\leq
p^{9\eta}
\sum_{x, \xi \in \fp}
\nuhat_k(\xi)^{2}
\nuhat_k(\xi x)^{2}
\nuk(x)
$$
\end{proof}
The rest of the proof is now essentially the same as the proof of
Theorem~\ref{t:main}.


\providecommand{\bysame}{\leavevmode\hbox to3em{\hrulefill}\thinspace}
\providecommand{\MR}{\relax\ifhmode\unskip\space\fi MR }
\providecommand{\MRhref}[2]{%
  \href{http://www.ams.org/mathscinet-getitem?mr=#1}{#2}
}
\providecommand{\href}[2]{#2}

\end{document}